\documentclass[11pt]{article}
\usepackage[square,numbers]{natbib}
\usepackage{color,latexsym,amsthm,amsmath,amssymb,path,enumitem,url,bbm}

\newcommand{\ignore}[1]{}

\newtheorem{theorem}{Theorem}[section]
\newtheorem{lemma}[theorem]{Lemma}

\newtheorem{claim}[theorem]{Claim}

\newcommand{\Proof}[1]
        {
        \noindent
        \emph{Proof #1.}~
        }
\newsavebox{\smallProofsym}                     
\savebox{\smallProofsym}                        %
        {
        \begin{picture}(7,7)                    %
        \put(0,0){\framebox(6,6){}}             %
        \put(0,2){\framebox(4,4){}}             %
        \end{picture}                           %
        }                                       %
\newcommand{\smalleop}[1]
        {
        \mbox{} \hfill #1~~\usebox{\smallProofsym}\!\!\!\!\!\!\
        }

\newcommand{\parag}[1]{\vspace{2mm}

\noindent{\bf #1} }

\usepackage{graphicx,psfrag}
\usepackage[rflt]{floatflt}

\usepackage{dsfont}

\newcommand{\ZZ}{\ensuremath{\mathbb Z}}
\newcommand{\RR}{\ensuremath{\mathbb R}}

\newcommand{\pts}{\mathcal P}

\newcommand{\lines}{\mathcal L}

\newcommand{\GCD}{\mathrm{GCD}}

\def\eps{{\varepsilon}}

\begin{document}
\pagenumbering{arabic}

\title{A structural Szemer\'edi--Trotter Theorem for Cartesian Products\thanks{This research project was done as part of the 2020 NYC Discrete Math REU, supported by NSF awards DMS-1802059, DMS-1851420, DMS-1953141, and DMS-2028892.}}

\author{
Adam Sheffer\thanks{Department of Mathematics, Baruch College, City University of New York, NY, USA.
{\sl adamsh@gmail.com}. Supported by NSF award DMS-1802059 and by PSCCUNY award 63580.}
\and
Olivine Silier\thanks{California Institute of Technology, Pasadena, CA 91125.
{\sl osilier@caltech.edu.} Supported by the Lynn A. Booth and Kent Kresa SURF Fellowship}}

\date{}

\maketitle

\begin{abstract}
We study configurations of $n$ points and $n$ lines that form $\Theta(n^{4/3})$ incidences, when the point set is a Cartesian product. 
We prove structural properties of such configurations, such that there exist many families of parallel lines or many families of concurrent lines. 
We show that the line slopes have multiplicative structure or that many sets of $y$-intercepts have additive structure.
We introduce the first infinite family of configurations with $\Theta(n^{4/3})$ incidences.
We also derive a new variant of a different structural point--line result of Elekes.

Our techniques are based on the concept of line energy. 
Recently, Rudnev and Shkredov introduced this energy and showed how it is connected to point--line incidences. 
We also prove that their bound is tight up to sub-polynomial factors. 
\end{abstract}

\section{Introduction}

The Szemer\'edi--Trotter theorem \cite{ST83} is a central result in discrete geometry. 
The many variants and generalizations of this theorem are now considered as an entire subfield. 
The Szemer\'edi-Trotter theorem and its variants have a wide range of applications, in combinatorics, harmonic analysis, theoretical computer science, number theory, model theory, and more (for example, see \cite{bombieri2015problem,BD15,CGS16,katz2019improved,singer2021point}).
It is thus awkward that hardly anything is known about the cases where this theorem is asymptotically tight.  

Let $\pts$ be a set of points and let $\lines$ be a set of lines, both in $\RR^2$.
An \emph{incidence} is a pair $(p,\ell)\in \pts\times\lines$, such that the point $p$ is on the line $\ell$.
We denote the number of incidences in $\pts\times\lines$ as $I(\pts,\lines)$.
The Szemer\'edi--Trotter theorem provides an asymptotically tight upper bound for $I(\pts,\lines)$.

\begin{theorem}{\bf (Szemer\'edi and Trotter)} \label{th:SzemTrot}
Let $\pts$ be a set of $m$ points and let $\lines$ be a set of $n$ lines, both in $\RR^2$.
Then
\[ I(\pts,\lines) = O(m^{2/3}n^{2/3}+m+n). \]
\end{theorem}

The \emph{structural} Szemer\'edi--Trotter problem asks to characterize the point--line configurations that have an asymptotically maximal number of incidences.  
We focus on sets of $n$ points and $n$ lines that form $\Theta(n^{4/3})$ incidences.
Most of our results can be extended to $m$ point and $n$ lines.

Around the middle of the 20th century, Erd\H os \cite{erd86} discovered a configuration of $n$ points and $n$ lines that form 
$\Theta(n^{4/3})$ incidences. 
In this construction, the point set is a $\sqrt{n}\times\sqrt{n}$ section of the integer lattice $\ZZ^2$.
In the early 2000s, Elekes \cite{Elek02} discovered a simpler configuration of $n$ points and $n$ lines that form 
$\Theta(n^{4/3})$ incidences.
In Elekes's configuration, the point set is an $n^{1/3}\times n^{2/3}$ section of the integer lattice $\ZZ^2$.

To obtain more point-line configurations, we can take one of the above configurations and perform transformations that preserve most of the incidences. 
We can apply projective transformations, point--line duality, remove some points and lines, and add additional points and lines. 
Up to such transformations, the only known configurations with $\Theta(n^{4/3})$ incidences were Erd\H os's and Elekes's.
This led to questions such as\begin{itemize}[noitemsep,topsep=1pt]
\item Are these the only two configurations? 
\item If there are more configurations, are there only sporadic configurations? Or is there an infinite family of configurations with continuous parameters?
\item Up to the above transformations, is the point set always a lattice? 
\item What properties must the set of lines satisfy?
\end{itemize}

Hardly anything is known about such structural questions.
The only non-trivial result that we are aware of is by Solymosi \cite{Solymosi06}.

\begin{theorem} \label{th:Soly}{\bf (Solymosi)}
For every constant integer $k$, the following holds for every sufficiently large $n$. Let $\pts$ be a set of $n$ points and let $\lines$ be a set of $n$ lines, both in $\RR^2$, such that $I(\pts,\lines)=\Theta(n^{4/3})$. Then there exists a set of $k$ of the points, no three on a line, such that there is a line of $\lines$ passing through each of the $\binom{k}{2}$ point pairs.
\end{theorem}

A recent result of Mirzaei and Suk \cite{MirzaeiSuk21} can be seen as quantitative variant of Theorem \ref{th:Soly}, for specific types of subgraphs.
A recent result of Hanson, Roche-Newton, and Zhelezov \cite[Theorem 1.8]{HROZ20} considers the case where the point set is $A\times A$, where $A$ is a set of rational numbers with a very small product set.
They show that, in this case, the number of incidences is significantly smaller than $n^{4/3}$. 

\parag{Our structural results.}
In this work, we prove several results for the structural Szemer\'edi--Trotter problem.
First, we provide the first infinite family of point--line configurations with $\Theta(n^{4/3})$ incidences. 
The constructions of Erd\H os and Elekes are the two extreme cases of this family.

\begin{theorem} \label{th:FamilyConst}
The following holds for every $1/3<\alpha<1/2$.
Let $\pts$ be a section of the integer lattice $\ZZ^2$ of size $n^\alpha \times n^{1-\alpha}$.
Then there exists a set $\lines$ of $n$ lines such that 
\[ I(\pts,\lines) = \Theta(n^{4/3}). \]
\end{theorem}

For a proof of Theorem \ref{th:FamilyConst} and additional details, see Section \ref{sec:constructions}. 
This family of constructions answers some structural questions and leads to new ones.
For example, in all of the constructions of this family, there are $\Theta(n^{1/3})$ families of $\Theta(n^{2/3})$ parallel lines. 
We wonder whether there exist configurations with $\Theta(n^{4/3})$ incidences that do not have this property, possibly after applying some transformations. 

A follow-up work by Larry Guth and the second author \cite{GuthSilier21} provides constructions where the point set is a Cartesian product of generalized arithmetic progressions.  
This prompts the question: Do all configurations consist of a Cartesian product of sets with a small sum set? For the definition of a sum set and more information, see Section \ref{sec:Energies}.  

We derive properties of the set of lines in the case where the point set is a Cartesian product.
Part (b) of the this result relies on additive and multiplicative energies.
For a definition of these energies, see Section \ref{sec:Energies}.

\begin{theorem} \label{th:StructuralSzemTrot} $\qquad$\\
(a) For $1/3<\alpha< 1/2$, let $A,B\subset \RR$ satisfy $|A|=n^\alpha$ and $|B|=n^{1-\alpha}$.
Let $\lines$ be a set of $n$ lines in $\RR^2$, such that $I(A\times B,\lines) = \Theta(n^{4/3})$.
Then at least one of the following holds:
\begin{itemize}
\item There exists $1-2\alpha\le \beta \le 2/3$ such that $\lines$ contains $\Omega(n^{1-\beta}/\log n)$ families of $\Theta(n^\beta)$ parallel lines, each with a different slope.
\item There exists $1-\alpha\le \gamma \le 2/3$ such that $\lines$ contains $\Omega(n^{1-\gamma}/\log n)$ disjoint families of $\Theta(n^\gamma)$ concurrent lines.
\end{itemize}
(b) Assume that we are in the case of $\Omega(n^{1-\beta}/\log n)$ families of $\Theta(n^\beta)$ parallel lines. 
There exists $n^{2\beta} \le t \le n^{3\beta}$ such that, for $\Omega(n^{1-\beta}/\log^2 n)$ of these families, the additive energy of the $y$-intercepts is $\Theta(t)$.
Let $S$ be the set of slopes of these families. 
Then
\[ E^{\times} (S) \cdot t = \Omega(n^{3-\alpha}/ \log^{12}n). \]
\end{theorem}

We can think of Theorem \ref{th:StructuralSzemTrot}(b) as another split into two cases: Either the set of slopes has a large multiplicative energy or many sets of $y$-intercepts have large additive energies. 
More intuitively and less accurately, either the set of slopes is somewhat similar to a geometric progression, or many sets of $y$-intercepts are somewhat similar to arithmetic progressions.
In all of the configurations that we are aware of, there are $\Theta(n^{1/3})$ sets of $\Theta(n^{2/3})$ parallel lines, the energy $E^{\times} (S)$ is small, and the sets of $y$-intercepts have large additive energies. 
One possibility is that all of the constructions have these properties.
For more information, see Section \ref{sec:constructions}.

For the proof of Theorem \ref{th:StructuralSzemTrot}, see Section \ref{sec:MainStructural}.
A follow-up work of the first author and Junxuan Shen \cite{Shen21} shows that, for lattices and several other cases, the concurrent case of Theorem \ref{th:StructuralSzemTrot}(a) cannot happen. 
Thus, if the concentric case can happen, then it takes place with very different point sets. 

The current work contains several additional related results, which are stated in the part titled \emph{Additional results} below.

\parag{Line energy.}
Not knowing much about the structural Szemer\'edi--Trotter problem is part of a more general phenomena.
We do not know much about the structural variants of most of the related problems. 
For example, not much is known about the structural distinct distances theorem (for more information, see \cite{Sheffer14}).
The main exception is the characterization of sets with few ordinary lines by Green and Tao \cite{GT13}.
In that paper, Green and Tao find structure by relying on tools from additive combinatorics. 
In the current work, we derive other structural results using different tools from additive combinatorics.

Elekes \cite{Elekes97b,Elekes98,Elekes99} considered a line $\ell$ that is defined by $y=ax+b$ as the linear function $f_\ell(x)=ax+b$.
He was then able to compose lines and to consider the inverse of a line.
This allowed Elekes to derive new combinatorial properties of lines. 
Recently, Rudnev and Shkredov \cite{RS18} further developed Elekes's ideas.
Given a set $\lines$ of non-axis-parallel lines, they consider the quantity
\[ \left|\left\{(\ell_1,\ell_2,\ell_3,\ell_4)\in \lines^4\ :\ f_{\ell_1}^{-1}\circ f_{\ell_2} = f_{\ell_3}^{-1}\circ f_{\ell_4} \right\} \right|. \]
We refer to this quantity as the \emph{line energy} of $\lines$ and denote it as $E(\lines)$.
A detailed and rigorous discussion of line energy can be found in Section \ref{sec:Energies}.

Rudnev and Shkredov derived an interesting connection between point-line incidences and line energy.

\begin{theorem} \label{th:IncidencesViaLineEnergy}
{\bf (Rudnev and Shkredov)} Let $A, B \subset \RR$ be finite sets and let $\lines$ be a set of non-axis-parallel lines in $\RR^2$.
Then
\[ I(A \times B,\lines) = O\left(|B|^{1/2}|A|^{2/3}E(\lines)^{1/6}|\lines|^{1/3} + |B|^{1/2}|\lines|\right). \]
\end{theorem}

Rudnev and Shkredov also derived upper bounds for $E(\lines)$.
The following upper bound was derived afterwards by Petridis, Roche-Newton, Rudnev, and Warren \cite{PRNRW20}.

\begin{theorem} \label{th:PrevEnergyBounds}
{\bf (Petridis, Roche-Newton, Rudnev, and Warren)} Let $\lines$ be a set of non-axis-parallel lines in $\RR^2$, such that at most $m$ lines have the same slope and every point is incident to at most $M$ lines.
Then 
\[ E(\lines) = O\left(m^{1/2}|\lines|^{5/2} + M|\lines|^2\right). \]
\end{theorem}

The paper \cite{PRNRW20} contains a dual formulation of Theorem \ref{th:PrevEnergyBounds}.
When reading that dual formulation, it may seem as if we need to add the term $m|\lines|^2$ to the bound of Theorem \ref{th:PrevEnergyBounds}. 
However, since $m\le |\lines|$, we have that $m|\lines|^2\le m^{1/2}|\lines|^{5/2}$.

\parag{Additional results.}
In Section \ref{sec:constructions}, we show that the main term of Theorem \ref{th:IncidencesViaLineEnergy} is tight up to sub-polynomial factors.

\begin{theorem} \label{th:tightRS}
For every $\eps>0$, in the first term of the bound of Theorem \ref{th:IncidencesViaLineEnergy}, no exponent could be decreased by an $\eps$.
\end{theorem}

Elekes \cite{Elekes99} studies lines that contain many points of a Cartesian product. 

\begin{theorem} \label{th:ElekesParCon}
{\bf (Elekes)} There exists a constant $c>0$ that satisfies the following for every $n$.
Let $A$ be a set of $n$ reals and let $0<\alpha<1$.
Let $\lines$ be a set of non-axis-parallel lines, each incident to at least $\alpha n$ points of $A\times A$.
Then at least one of the following holds:
\begin{itemize}
\item There exist $\alpha^c n$ lines of $\lines$ that are parallel.
\item There exist $\alpha^c n$ lines of $\lines$ that are concurrent. 
\end{itemize} 
\end{theorem}

Petridis, Roche-Newton, Rudnev, and Warren \cite{PRNRW20} derived a quantitative variant of Theorem \ref{th:ElekesParCon}.

\begin{theorem} \label{th:ElekesQuantitative}
{\bf (Petridis, Roche-Newton, Rudnev, and Warren)} Let $A$ be a set of $n$ reals and let $\alpha<1$ satisfy $\alpha = \Omega(n^{-1/2})$.
Let $\lines$ be a set of $k$ non-axis-parallel lines, each incident to at least $\alpha n$ points of $A\times A$.
Then at least one of the following holds:
\begin{itemize}
\item There exist $\alpha^{12} n^{-2}k^3$ lines of $\lines$ that are parallel and $E^+(A)=\Omega(\alpha^{14}k^3)$.
\item There exist $\alpha^{6} n^{-1}k^2$ lines of $\lines$ that are concurrent. Also, there exists $s$ such that $E^\times(A-s)=\Omega(\alpha^8nk^2)$. (The set $A-s$ is obtained by subtracting $s$ from every element of $A$.)
\end{itemize}
\end{theorem}

While Theorems \ref{th:ElekesParCon} and \ref{th:ElekesQuantitative} seem similar to Theorem \ref{th:StructuralSzemTrot}, there are several significant differences between these two structural problems. 
First, it is not difficult to find a configuration that satisfies the concurrent case of Theorem \ref{th:ElekesParCon}. 
On the other hand, it is plausible that the concurrent case of Theorem \ref{th:StructuralSzemTrot} does not exist. 

To see another difference between the two above scenarios, we apply Theorem \ref{th:ElekesQuantitative} to the structural problem of Theorem \ref{th:StructuralSzemTrot}. 
When $n$ points and $n$ lines form $\Theta(n^{4/3})$ incidences, those incidences originate from $\Theta(n)$ lines that are incident to $\Theta(n^{1/3})$ points (see the proof of Lemma \ref{le:ThirdRichness} below). 
In our case there are $|A|^2=n^2$ points, so $\Theta(n^2)$ lines are incident to $\Theta(n^{2/3})$ points. 
That is, $k=n^2$ and $\alpha = \Theta(n^{-1/3})$.
Then, Theorem \ref{th:ElekesQuantitative} leads to the trivial statement that there exist $\Omega(1)$ parallel lines.
While the two structural problems have a different behavior, the proofs of Theorem \ref{th:ElekesQuantitative} and of Theorem \ref{th:StructuralSzemTrot}(a) both rely on Theorem \ref{th:IncidencesViaLineEnergy} and Theorem \ref{th:PrevEnergyBounds}.

We introduce a variant of Theorem \ref{th:ElekesQuantitative} that does lead to interesting results in the case of $\Theta(n^{4/3})$ incidences.
This variant is in the style of Theorem \ref{th:StructuralSzemTrot}(b).
By the pigeonhole principle, there exists $0\le \beta\le 1$ such that $\lines$ contains $\Omega(k^{1-\beta}/\log k)$ families of $\Theta(k^\beta)$ parallel lines, each with a different slope.
We provide an upper bound on $k$ in terms of the multiplicative energy of these slopes and of the additive energy of their $y$-intercepts.

\begin{theorem} \label{th:ElekesStructural}
Let $A$ be a set of $n$ reals and let $\alpha<1$ satisfy $\alpha = \Omega(n^{-1/2}\log^2 k)$.
Let $\lines$ be a set of $k$ non-axis-parallel lines, each incident to at least $\alpha n$ points of $A\times A$.
Consider $0\le\beta\le 1$ such that $\lines$ contains $\Omega(k^{1-\beta}/\log k)$ families of $\Theta(k^\beta)$ parallel lines.
There exists $k^{2\beta} \le t \le k^{3\beta}$ such that, for $\Omega(k^{1-\beta}/\log^2 k)$ of these families, the additive energy of the $y$-intercepts is $\Theta(t)$.
Let $S$ be the set of slopes of these families. 
Then 
\[ k = O\left(\frac{n^{1/4}t^{1/4}E^\times(S)^{1/4}\log^{3} n}{\alpha^{3/2}}\right). \]
\end{theorem}

The statement of Theorem \ref{th:ElekesStructural} is long and technical, but it also has a simple intuition:
If many lines are incident to many points of $A\times A$, then the slopes have a large multiplicative energy or the $y$-intercepts of many parallel families have large additive energies.
For a proof of Theorem \ref{th:ElekesStructural}, see Section \ref{sec:ElekStructural}.

To see that Theorem \ref{th:ElekesStructural} provides a reasonable bound in the case of $\Theta(n^{4/3})$ incidences, we consider Erd\H os's construction with $n^2$ points and lines (see for example \cite{ShefferBlog16}).
Like all of the known constructions, the set of lines contains $n^{2/3}$ families of $n^{4/3}$ parallel lines. 
Each of those lines forms $\Theta(n^{2/3})$ incidences, so $\alpha = n^{-1/3}$.
A simple variant of our proof of Lemma \ref{le:SlopesMultEnergy} shows that $E^\times(S) = O(n^{4/3+\eps})$, for every $\eps>0$.
The sets of $y$-intercepts are arithmetic progressions, so $t=\Theta(n^4)$.
Plugging this into Theorem \ref{th:ElekesStructural} leads to the bound $k=O(n^{25/12+\eps})$.
This is not far from the correct bound $k=\Theta(n^2)$, so Theorem \ref{th:ElekesStructural} cannot be significantly improved.

For the opposite case of Theorem \ref{th:ElekesStructural}, we set $A=\{2^0,2^1,2^2,\ldots,2^{n-1}\}$ and consider the lines that are defined by $y=2^j\cdot x$ with $j\in \{0,1,2,...,n/2\}$.
In this case, there are $k$ families of a single parallel line, so $t=1$.
Every line contains $\Theta(n)$ points of $A\times A$, so $\alpha =\Theta(1)$.
The line slopes form a geometric progression, so $E^\times(S) = \Theta(n^3)$.
Plugging the above into Theorem \ref{th:ElekesStructural} leads to the bound $k=O(n\log^3 n)$, which is tight up to the polylogarithmic factor.

Many of the results in this paper can be extended to other fields, to $m$ points and $n$ lines, and to other problems.

\parag{Acknowledgments.}
The authors thank Misha Rudnev, Junxuan Shen, Ilya Shkredov, and Audie Warren for helpful conversations.
We also thank Chi Hoi Yip and the anonymous referees for spotting issues and helping to improve this work. 

\section{Preliminaries}

In this section we introduce a few tools that we use in our proofs. 
We suggest to quickly skim this section and return to it when necessary. 

\parag{Euler's totient function.}
For a positive integer $b$, the \emph{Euler totient function} of $b$ is
\[ \phi(b) = \left| \{\ a\ :\ 1\le a < b \quad \text{ and } \quad \GCD(a,b)=1 \}\right|.  \]

The following lemma collects a few basic properties of the totient function.

\begin{lemma} \label{le:totient} 
Let $r$ be a positive integer. Then \\
(a) $\displaystyle \sum_{i=1}^r \phi(i) = \frac{3}{\pi^2}r^2+O(r\log r)$, \\[2mm]
(b) $\displaystyle \sum_{i=1}^r i\cdot \phi(i) = \Theta(r^3)$, \\[2mm]
(c) $\displaystyle \sum_{i=1}^r \phi(i)/i^2 = O(\log r)$. 
\end{lemma}
\begin{proof}
Part (a) is a classic formula. For example, see \cite[Theorem 330]{HW79}.
For part (b), we note that
\begin{align*} 
\sum_{i=1}^r i\cdot \phi(i) &\le r\cdot \sum_{i=1}^r \phi(i) = O(r^3), \quad \text{ and that } \\
 \sum_{i=1}^r i\cdot \phi(i) &> \sum_{i=r/2}^r i\cdot \phi(i) > (r/2)\left(\sum_{i=1}^r \phi(i) - \sum_{i=1}^{r/2} \phi(i)\right) = \Omega(r^3).
\end{align*}
Combining the above bounds implies that $\sum_{i=1}^r i\cdot \phi(i) = \Theta(r^3)$.

By definition, $\phi(i)\le i$ holds for every positive integer $i$.
This implies that 
\[ \sum_{i=1}^r \phi(i)/i^2 \le \sum_{i=1}^r 1/i = O(\log n). \] 
\end{proof}

\parag{Rich lines.}
Let $\pts$ be a set of points in $\RR^2$. 
For an integer $r\ge 2$, a line $\ell\subset \RR^2$ is $r$-\emph{rich} with respect to $\pts$ if $\ell$ is incident to at least $r$ points of $\pts$.
The following result is often called a \emph{dual} Szemer\'edi--Trotter theorem.
This is because each of Theorems \ref{th:SzemTrot} and \ref{th:DualSzemTrot} can be derived from the other by only using elementary arguments. 

\begin{theorem} \label{th:DualSzemTrot}
Let $\pts$ be a set of $m$ points in $\RR^2$ and let $r$ be an integer larger than one.
Then, the number of $r$-rich lines with respect to $\pts$ is
\[ O\left(\frac{m^2}{r^3} + \frac{m}{r}\right).\]
\end{theorem}

\section{Energies} \label{sec:Energies}

In this section we describe line energy in more detail.
We begin with a brief description of additive and multiplicative energies.
An expert reader may wish to skip this part.

\parag{Additive and multiplicative energies.}
Let $A\subset \RR$ be a finite set.
The \emph{sum set} of $A$ is
\[ A+A = \{a+a'\ :\ a,a'\in A \}. \]
It is not difficult to verify that $|A+A|=\Omega(|A|)$ and $|A+A|=O(|A|^2)$.
Intuitively, a small sum set implies that $A$ is similar to an arithmetic progression. 
The \emph{additive energy} of $A$ is 
\[ E^+(A) =|\{(a,b,c,d)\in A^4\ :\ a+b=c+d \}|. \]

After fixing the values of $a,b$, and $c$, at most one $d\in A$ satisfies $a+b=c+d$.
This implies that $E^+(A)\le |A|^3$.
By considering the case where $a=c$ and $b=d$, we get that $E^+(A)\ge |A|^2$.

Additive energy is a central object in additive combinatorics.
Intuitively, it provides information about the structure of $A$ under addition. 
We now consider one example of this.
For $x\in \RR$, we define 
\begin{align*} 
r_A^+(x) &= |\{(a,a')\in A^2\ :\ a+a'=x\}|, \\[2mm]
r_A^-(x) &= |\{(a,a')\in A^2\ :\ a-a'=x\}|.
\end{align*}

Since every pair from $A^2$ contributes to exactly one $r_A^+(x)$, we have that $\sum_{x\in \RR} r_A^+(x) = |A|^2$.
For a fixed $x\in \RR$, the number of solutions from $A$ to $a+b=c+d=x$ is $r_A^+(x)^2$.
This implies that $E^+(A) = \sum_{x} r_A^+(x)^2$.
Then, the Cauchy-Schwarz inequality leads to
\[ E^+(A) = \sum_{x\in A+A} r_A^+(x)^2 \ge \frac{\left(\sum_{x\in A+A} r_A^+(x) \right)^{2}}{\sum_{x\in A+A} 1} = \frac{(|A|^2)^2}{|A+A|} = \frac{|A|^4}{|A+A|}. \]
Thus, a small sum set implies a large additive energy. 
In particular, if $|A+A|=\Theta(|A|)$ then $E^+(A)=\Omega(|A|^3)$.


For finite $A,B\subset \RR$, we define the additive energy of $A$ and $B$ as
\[ E^+(A,B) = |\{(a,b,c,d)\in A\times B\times A \times B\ :\ a+b=c+d \}|. \]
We can rearrange the above condition as $a-c=d-b$.
Then, the Cauchy-Schwarz inequality implies that 
\begin{align} 
E^+(A,B) = \sum_{x\in \RR} r_A^-(x) \cdot r_B^-(x) &\le \Big(\sum_{x\in \RR} r_A^-(x)^2\Big)^{1/2} \cdot \Big(\sum_{x\in \RR} r_B^-(x)^2\Big)^{1/2} \nonumber \\[2mm]
&\hspace{30mm} = E^+(A)^{1/2}\cdot E^+(B)^{1/2}. \label{eq:BipartiteEnergy}
\end{align}

The \emph{product set} of $A$ is
\[ AA = \{a\cdot a'\ :\ a,a'\in A \}. \]
The \emph{multiplicative energy} of $A$ is 
\begin{equation} \label{eq:MultEnergy} 
E^\times(A) =|\{(a,b,c,d)\in A^4\ :\ a\cdot b=c\cdot d \}|. 
\end{equation}

Product sets and multiplicative energies are similar to sum sets and additive energies.
For example, a small product set implies a large multiplicative energy.
Intuitively, a set with a small product set is similar to a geometric progression.

\parag{Line energy.}
In the following, when considering lines, we refer only to non-axis-parallel lines in $\RR^2$.
We associate a line $\ell$ with the function $f_\ell(x)=cx+d$.
The \emph{composition} of the lines $\ell$ and $\ell'$ that are defined by $y=cx+d$ and $y=c'x+d'$ is
\[ \ell \circ \ell' = f_\ell\circ f_{\ell'} = c\cdot f_{\ell'}(x) + d = cc'x+cd'+d. \]

The set of non-axis-parallel lines under the above composition form the \emph{affine group} of $\RR$.
However, for our purposes, we prefer to think of the group elements as lines in $\RR^2$. 
We also associate a line $\ell$ that is defined by $y=cx+d$ with the point $(c,d)\in \RR^2$.
We refer to the plane that contains $\ell$ as the \emph{primal plane} and to the plane that contains $(c,d)$ as the \emph{dual plane}.
With this dual notation, the group operation is 
\begin{equation} \label{eq:GroupOper} 
(c,d) \circ (c',d') = (cc',cd'+d). 
\end{equation}

The identity element is $(1,0)$ and the inverse of $(c,d)$ is $(1/c,-d/c)$.
This group is non-commutative. 
For example, we have that
\[ (2,3) \circ (2,4) = (4,11) \neq (4,10) = (2,4) \circ (2,3). \]

We define the \emph{line energy} of a set of lines $\lines$ as
\[ E(\lines) = \left|\left\{(\ell_1,\ell_2,\ell_3,\ell_4)\in \lines^4\ :\ f_{\ell_1}^{-1}\circ f_{\ell_2} = f_{\ell_3}^{-1}\circ f_{\ell_4} \right\} \right|. \]
When writing $f_{\ell_j} = c_jx+d_j$, we also have the dual formulation
\[ E(\lines) = \left|\left\{(\ell_1,\ell_2,\ell_3,\ell_4)\in \lines^4\ :\ (c_1,d_1)^{-1}\circ (c_2,d_2) = (c_3,d_3)^{-1}\circ (c_4,d_4) \right\} \right|. \]

Rudnev and Shkredov \cite{RS18} refer to $E(\lines)$ as an energy, but do not give it an explicit name.
The recent paper \cite{PRNRW20} refers to $E(\lines)$ as \emph{affine energy}, because of its connection with the affine group. 
We use the name line energy, since think of $E(\lines)$ as a property of a set of lines. 

We note that 
\[ (c,d)^{-1}\circ (c',d') = \left(\frac{1}{c},\frac{-d}{c}\right)\circ (c',d') = \left(\frac{c'}{c}, \frac{d'-d}{c}\right). \]
Since we only consider non-axis-parallel lines, $c\neq0$ and the above is well-defined. 
A quadruple $(\ell_1,\ell_2,\ell_3,\ell_4)\in \lines^4$ contributes to $E(\lines)$ if and only if 
\[ \left(\frac{c_2}{c_1}, \frac{d_2-d_1}{c_2}\right) = \left(\frac{c_4}{c_3}, \frac{d_4-d_3}{c_4}\right). \]
Simplifying leads to the system 
\begin{align}  
c_2 \cdot c_3 &= c_1\cdot c_4, \label{eq:LineQuadrupleCond1} \\
c_3(d_2-d_1) &= c_1(d_4-d_3). \label{eq:LineQuadrupleCond2}
\end{align}

\subsection{Cartesian products of lines} 
In the dual plane, we can consider a set of lines that is a Cartesian product.
That is, a Cartesian product $C\times D$ consists of the points that are dual to lines with a slope from $C$ and a $y$-intercept from $D$.
The following bound for Cartesian products of lines is a warm-up towards similar arguments that we use in the following sections. 

\begin{theorem} \label{th:LineCartesianEnergyUpper}
Consider finite sets $C,D\subset \RR\setminus\{0\}$.
Let $\lines$ be a set of lines in $\RR^2$ that is dual to $C\times D$.  
Then
\begin{equation} \label{eq:TwoEnergiesBound} 
E(C\times D) \le E^\times(C)\cdot E^+(D). 
\end{equation}
\end{theorem}
\begin{proof}
The number of solutions to \eqref{eq:LineQuadrupleCond1} is $E^\times(C)$.
We fix a solution to \eqref{eq:LineQuadrupleCond1} and derive an upper bound on the number of corresponding solutions to \eqref{eq:LineQuadrupleCond2}.
That is, we fix $c_1,c_2,c_3,c_4$, and consider the number of valid values for $d_1,d_2,d_3,d_4$.

We set $s = c_1/c_3$ and rephrase \eqref{eq:LineQuadrupleCond2} as 
\begin{equation} \label{eq:LineQuadrupleCond2rev} 
d_2-d_1 = s(d_4-d_3).
\end{equation}

By the Cauchy-Schwarz inequality, the number of solutions to \eqref{eq:LineQuadrupleCond2rev} is 
\begin{align*} 
\sum_{k\in \RR}r^-_D(k) r^-_D(k/s) &\le \left(\sum_{k\in \RR}r^-_D(k)^2\right)^{1/2} \left(\sum_{k\in \RR}r^-_D(k/s)^2\right)^{1/2} \\[2mm]
&\hspace{50mm} = \left(E^+(D)\right)^{1/2} \left(E^+(D)\right)^{1/2} = E^+(D).
\end{align*}

There are $E^\times(C)$ solutions to \eqref{eq:LineQuadrupleCond1} and each corresponds to at most $E^+(D)$ solutions to \eqref{eq:LineQuadrupleCond2}. This implies that $E(C\times D) \le E^\times(C)\cdot E^+(D)$.
\end{proof}

To put Theorem \ref{th:LineCartesianEnergyUpper} in context, we consider Elekes's construction from \cite{Elek02}. 
In particular, we consider the point set
\[ \pts = \left\{\ (i,j)\ :\ 1 \le i \le n^{1/3}/2 \quad \text{ and } \quad 1 \le j \le 2n^{2/3} \ \right\}, \]
and the set of lines
\begin{equation} \label{eq:ElekesLines} 
\lines = \left\{\ y=ax+b \ :\ 1 \le a \le n^{1/3} \quad \text{ and } \quad 1 \le b \le n^{2/3} \ \right\}. 
\end{equation}

We note that $|\pts|=|\lines|=n$, that $I(\pts,\lines)=\Theta(n^{4/3})$, and that the set of lines is a Cartesian product.

\begin{claim} \label{cl:ElekesEnergy}
The set $\lines$ from \eqref{eq:ElekesLines} satisfies that $E(\lines) = O(n^{8/3}\log n)$.
This is tight, possibly up to the logarithmic factor. 
\end{claim}
\begin{proof}
We write $C=\{1,2,\ldots,n^{1/3}\}$ and $D=\{1,2,\ldots,n^{2/3}\}$.
Then $\lines$ is dual to the Cartesian product $C\times D$.
We define $r_C^\div(x) = |\{(c,c')\in C^2\ :\ c/ c' =x\}|$.
Since $0\notin C$, the condition $a\cdot b = c\cdot d$ from \eqref{eq:MultEnergy} is equivalent to $a/c=d/b$.
This implies that $E^\times(C) = \sum_x r_C^\div(x)^2$.
We note that $r_C^\div(1)^2 = n^{2/3}$ and that $r_C^\div(x) = r_C^\div(1/x)$ for every $x$.
This implies that 
\begin{align*}
E^\times(C) &= O(n^{2/3})+ 2\sum_{s=1}^{n^{1/3}}\sum_{1\le t <s \atop \GCD(s,t)=1} r_C^\div(t/s)^2 = O(n^{2/3})+ 2 \sum_{s=1}^{n^{1/3}}\sum_{1\le t <s \atop \GCD(s,t)=1}  \left(\frac{n^{1/3}}{s}\right)^2. 
\end{align*}

In the final transition above, we note that every $(c,c')\in C$ that satisfy $c/c' = t/s$ can be written as $c=at$ and $c'=as$ where $1\le a\le n^{1/3}/s$.
Thus, the number of representations of $t/s$ as $c/c'$ is $\lfloor n^{1/3}/s \rfloor$.
For a fixed $s$, the sum over $t$ contains $\phi(s)$ identical elements.
We may thus write
\begin{align*} 
E^\times(C) \le O(n^{2/3})+ 2 \sum_{s=1}^{n^{1/3}} \phi(s) \left(\frac{n^{1/3}}{s}\right)^2 = O(n^{2/3})+ 2 n^{2/3}\sum_{s=1}^{n^{1/3}} \frac{\phi(s)}{s^2} = O(n^{2/3}\log n).
\end{align*}
In the last transition, we applied Lemma \ref{le:totient}(c).

Since $D$ is an arithmetic progression, we have that $E^+(D) = \Theta(|D|^3)= \Theta(n^2)$. 
Theorem \ref{th:LineCartesianEnergyUpper} implies that 
\[ E(\lines)\le E^\times(C)\cdot E^+(D) = O(n^{2/3}\log n) \cdot\Theta(n^2) = O(n^{8/3}\log n). \]

To show that the above bound is close to tight, we claim that there are many solutions to \eqref{eq:LineQuadrupleCond1} and \eqref{eq:LineQuadrupleCond2}.
We rephrase \eqref{eq:LineQuadrupleCond1} as $c_2/c_4 = c_1/c_3$.
There are $n^{2/3}$ solutions in $C$ to $c_2/c_4 = c_1/c_3=1$. 
For each of these solutions, \eqref{eq:LineQuadrupleCond2} becomes $d_2-d_1 = d_4-d_3$.
The number of solutions of this equation is $E^+(D) = \Theta(n^2)$.
Thus, there are $\Omega(n^{8/3})$ solutions to \eqref{eq:LineQuadrupleCond1} and \eqref{eq:LineQuadrupleCond2}.
In other words, $E(\lines)=\Omega(n^{8/3})$.
\end{proof}

Claim \ref{cl:ElekesEnergy} shows that Theorem \ref{th:LineCartesianEnergyUpper} is tight, possibly up to a logarithmic factor. 
Combining Claim \ref{cl:ElekesEnergy} with Theorem \ref{th:IncidencesViaLineEnergy} leads to
\[ I(\pts,\lines) = O\left(n^{1/3}n^{2/9}(n^{8/3}\log n)^{1/6}n^{1/3} + n^{1/3}n\right) = O(n^{4/3}\log^{1/6} n). \]
This implies that the bound of Theorem \ref{th:IncidencesViaLineEnergy} is tight, possibly up to a logarithmic factor. 
However, it is possible that this tightness is achieved by the less interesting term $|B|^{1/2}|\lines|$.
In Theorem \ref{th:tightRS}, we show that the main term of the bound of Theorem \ref{th:IncidencesViaLineEnergy} is tight up to sub-polynomial factors. 

\section{New constructions} \label{sec:constructions}

In this section we prove Theorem \ref{th:FamilyConst} and Theorem \ref{th:tightRS}.
We first recall the statement of each theorem.
\vspace{2mm}

\noindent {\bf Theorem \ref{th:FamilyConst}.}
\emph{The following holds for every $1/3<\alpha<1/2$.
Let $\pts$ be a section of the integer lattice $\ZZ^2$ of size $n^\alpha \times n^{1-\alpha}$.
Then there exists a set $\lines$ of $n$ lines such that }
\[ I(\pts,\lines) = \Theta(n^{4/3}). \]
\begin{proof}
We may assume that 
\begin{equation} \label{eq:ConstFamilyPts}
    \pts = \bigg\{ (i,j)\ :\ 0 \leq i < n^{\alpha} \quad  \text{  and  } \quad -n^{1-\alpha}/2 < j \leq n^{1-\alpha}/2 \bigg\}.
\end{equation}
If $\pts$ is a different section of $\ZZ^2$ of size $n^\alpha \times n^{1-\alpha}$, then we obtain \eqref{eq:ConstFamilyPts} after a translation of $\RR^2$.
Such a translation does not affect the number of incidences.

We consider the set of lines 
\begin{align}
    \lines = \bigg\{ y = \bigg(a + \frac{b}{c}&\bigg) \cdot (x-i) + d\ :\ (a,b,c,d) \in \mathbb{N}^4,\ 0 \leq a < n^{1-2\alpha}/4, \nonumber \\
    & 1 \leq b <c \leq n^{\alpha - 1/3},\ \GCD(b,c) = 1,\ 0 \leq i < c,\ 0 \leq d < n^{1-\alpha}/4 \bigg\}. \label{eq:GenLinesConst}
\end{align}

In \eqref{eq:GenLinesConst}, after fixing $1 <c \leq n^{\alpha - 1/3}$, there are $c$ possible values for $i$ and $\phi(c)$ possible values for $b$. 
Lemma \ref{le:totient}(b) implies that 
\[ |\lines| = \sum_{c=2}^{n^{\alpha-1/3}} \Big( \frac{n^{1-2\alpha}}{4}\cdot \frac{n^{1-\alpha}}{4} \cdot c \cdot \phi(c)\Big) = \frac{n^{2-3\alpha}}{16}\cdot\sum_{c=2}^{n^{\alpha-1/3}} (c \cdot \phi(c)) = \frac{n^{2-3\alpha}}{16} \cdot \Theta((n^{\alpha - 1/3})^3) = \Theta(n). \]
At the end of the proof, we revise $\lines$ to ensure that $|\lines|=n$.

We consider $a,b,c,d,i$ that satisfy the restrictions in \eqref{eq:GenLinesConst} and also $0 \leq x < n^{\alpha}$.
Then 
\begin{align} 
y &= \left(a + \frac{b}{c}\right) (x-i) + d < \frac{n^{1-2\alpha}}{4} \cdot n^{\alpha} + \frac{n^{1-\alpha}}{4} = \frac{n^{1-\alpha}}{2}, \nonumber \\
y &= \left(a + \frac{b}{c}\right) (x-i) + d > \frac{n^{1-2\alpha}}{4}\cdot (-n^{\alpha-1/3}) = \frac{-n^{2/3-\alpha}}{4}. \label{eq:ConstYvalues}
\end{align}
Thus, $y$ is always in the range of the $y$-coordinates of the points from \eqref{eq:ConstFamilyPts}, although $y$ may not be an integer.

Consider a line $\ell\in \lines$.
We note that $\left(a + \frac{b}{c}\right) (x-i) + d$ is an integer for every $c$'th value of $x\in \{0,1,\ldots,n^{\alpha}-1\}$. 
This implies that $\ell$ is incident to a point from every $c$'th column of $\pts$. 
By combining this with \eqref{eq:ConstYvalues}, we get that $\ell$ is incident to $\Theta(n^{\alpha}/c)$ points of $\pts$. 

For $1 <c \leq n^{\alpha - 1/3}$, let $\lines_c$ be the set of lines of $\lines$ that are defined with this $c$.
By inspecting \eqref{eq:GenLinesConst}, we note that 
\[ |\lines_c| = \frac{n^{1-2\alpha}}{4} \cdot \frac{n^{1-\alpha}}{4} \cdot c \cdot \phi(c) = \frac{n^{2-3\alpha} \cdot c \cdot \phi(c)}{16}. \]
Lemma \ref{le:totient}(a) implies that 
\begin{align*}
    I(\pts, \lines) = \sum_{c=2}^{n^{\alpha - 1/3}} |\lines_c| \cdot \Theta\left(\frac{n^{\alpha}}{c}\right) &= \sum_{c=2}^{n^{\alpha - 1/3}}\Theta\left(n^{2-2\alpha}\cdot \phi(c)\right) \\
    &= \Theta\left(n^{2-2\alpha}\sum_{c=2}^{n^{\alpha - 1/3}}\phi(c)\right) = \Theta(n^{4/3}).
\end{align*}

We still need to change $\lines$ so that $|\lines|=n$.
We recall that $|\lines|=\Theta(n)$.
If $|\lines|< n$ then we add arbitrary lines to $\lines$ until $|\lines|=n$.
If $|\lines|> n$ then we repeatedly remove a line with the smallest number of incidences, until $|\lines|=n$.
In either case, the number of incidences is asymptotically unchanged.
\end{proof}

Repeating the above construction with $\alpha = 1/2$ leads to Erd\H{o}s's construction. 
Repeating it with $\alpha = 1/3$ leads to Elekes's construction.

The following lemma provides additional properties of the set of lines that was introduced in Theorem \ref{th:FamilyConst}.

\begin{lemma} \label{le:SlopesMultEnergy}
Let $S$ be the set of slopes of the lines from \eqref{eq:GenLinesConst}.
Then $|S| = \Theta(n^{1/3})$ and, for every $\eps>0$, we have that
\[ E^\times(S) = O(n^{2/3+\eps}). \]
\end{lemma}
\begin{proof}
By inspecting \eqref{eq:GenLinesConst}, we get that
\begin{equation} \label{eq:NewFamSlopes}
S=  \bigg\{ a + \frac{b}{c}\ :\ (a,b,c) \in \mathbb{N}^3,\ 0 \leq a < \frac{n^{1-2\alpha}}{4},\ 1 \leq b <c \leq n^{\alpha - 1/3},\ \GCD(b,c) = 1 \bigg\}. 
\end{equation}
Lemma \ref{le:totient}(a) implies that
\[ |S| = \frac{n^{1-2\alpha}}{4} \cdot \sum_{c=2}^{n^{\alpha-1/3}}\phi(c) = \Theta(n^{1-2\alpha} \cdot n^{2\alpha-2/3}) = \Theta(n^{1/3}). \]

Let $q$ be a positive integer.
Let $A$ be a set of rational numbers, where all the denominators and numerators are at most $q$.  
Theorem 3 of Shteinikov \cite{shteinikov17} states that
\[ E^{\times}(A) \le |A|^2\cdot q^{O(1/\log\log q)}. \]

We note that every element of $S$ can be written as a rational number with the denominator being at most $n^{\alpha-1/3}$ and the numerator at most $n^{2/3-\alpha}$.
Theorem 3 of Shteinikov \cite{shteinikov17} states that, for every $\eps>0$, 
\[ E^{\times}(S) \le |S|^2 \cdot n^{O(1/\log\log n)} = O(n^{2/3+\eps}). \]
\end{proof}

We are now ready to prove Theorem \ref{th:tightRS}.
\vspace{2mm}

\noindent {\bf Theorem \ref{th:tightRS}.}
\emph{For every $\eps>0$, in the first term of the bound of Theorem \ref{th:IncidencesViaLineEnergy}, no exponent could be decreased by an $\eps$.}
\begin{proof}
Let $\alpha = 1/3+\eps^*$, where the value of $\eps^*>0$ is determined below.
With this value of $\alpha$, we consider the point set $\pts$ from \eqref{eq:ConstFamilyPts} and the line set $\lines$ from \eqref{eq:GenLinesConst}.
We recall that $I(\pts,\lines)=\Theta(n^{4/3})$ and that $\pts$ is a section of the integer lattice of size $n^\alpha\times n^{1-\alpha}$.

Assume for contraction that the bound of Theorem \ref{th:IncidencesViaLineEnergy}(a) holds after decreasing one of the exponents of the first term by $\eps$.
We apply this improved bound on $\pts$ and $\lines$, to obtain that 
\[ I(\pts,\lines) = O\left(n^{(1-\alpha)/2}n^{2\alpha/3}E(\lines)^{1/6}n^{1/3}/n^{\eps/10} + n^{(1-\alpha)/2}n\right). \]

Since $\alpha>1/3$, the second term of the above bound is asymptotically smaller than $n^{4/3}$.
Since $I(\pts,\lines)=\Theta(n^{4/3})$, this term cannot dominate the bound.
This leads to 
\begin{equation*} 
n^{4/3} =O(n^{\frac{5+\alpha}{6}-\frac{\eps}{10}}\cdot E(\lines)^{1/6}). 
\end{equation*}
Rearranging and recalling that $\alpha = 1/3+\eps^*$ gives that 
\begin{equation} \label{eq:ShkredovRudnevTightnessUpperEnergy}
E(\lines) = \Omega(n^{3-\alpha+3\eps/5}) = \Omega(n^{8/3-\eps^*+3\eps/5}).
\end{equation}

To derive an upper bound on $E(\lines)$, we use a variant of the proof of Theorem \ref{th:LineCartesianEnergyUpper}.
Let $S$ be the set of slopes of $\lines$, as defined in \eqref{eq:NewFamSlopes}. 
Lemma \ref{le:SlopesMultEnergy} implies that $E^\times(S) = O(n^{2/3+\eps^*})$.
In other words, there are $O(n^{2/3+\eps^*})$ solutions to \eqref{eq:LineQuadrupleCond1} when $(c_1,c_2,c_3,c_4)\in S^4$.
By inspecting \eqref{eq:GenLinesConst}, we note that $O(n^{2/3})$ lines of $\lines$ have the same slope. 
Thus, for every solution of \eqref{eq:LineQuadrupleCond1}, there are $O(n^2)$ possible values for $d_1,d_2,d_3$. 
After fixing those, at most one value of $d_4$ satisfies \eqref{eq:LineQuadrupleCond2}.
This implies that 
\[ E(\lines) = O(n^{2/3+\eps^*}\cdot n^2) = O(n^{8/3+\eps^*}). \]

By setting $\eps^*<3\eps/10$, we get a contradiction to \eqref{eq:ShkredovRudnevTightnessUpperEnergy}. 
Indeed, in this case $8/3+\eps^*<8/3-\eps^*+3\eps/5$.
\end{proof}

\section{The main structural result} \label{sec:MainStructural}

In this section we prove Theorem \ref{th:StructuralSzemTrot}.
Our proof requires the following lemma. 

\begin{lemma} \label{le:ThirdRichness}
Consider a set $\pts$ of $n$ points and a set $\lines$ of $n$ lines, both in $\RR^2$, such that $I(\pts,\lines)=\Theta(n^{4/3})$.
Then there exists a subset $\lines' \subseteq \lines$ such that $|\lines'|=\Theta(n)$, $I(\pts,\lines')=\Theta(n^{4/3})$, and every line of $\lines'$ is incident to $\Theta(n^{1/3})$ points of $\pts$.
\end{lemma}
\begin{proof}
Recall that $o(\cdot)$ means ``asymptotically strictly smaller than."
We write $I(\pts,\lines)=c_1n^{4/3} + o(n^{4/3})$.
Let $\lines_1$ be the set of lines of $\lines$ that are incident to at most $c_1n^{1/3}/4$ points of $\pts$.
Then
\begin{equation} \label{eq:PoorLinesInc} 
I(\pts,\lines_1) \le \frac{c_1n^{1/3}}{4}\cdot |\lines_1| \le \frac{c_1n^{4/3}}{4}. 
\end{equation}

Let $c_2$ be the constant that is hidden by the $O(\cdot)$-notation in the bound of Theorem \ref{th:DualSzemTrot}.
Let $\lines_2$ be the set of lines of $\lines$ that are incident to at least $(10c_2n)^{1/3}/c_1^{1/2}$ points of $\pts$.
Theorem \ref{th:DualSzemTrot} implies that
\[ |\lines_2| \le \frac{c_2n^2}{10c_2n/c_1^{3/2}} + \frac{c_2n}{(10c_2n)^{1/3}/c_1^{1/2}} = \frac{c_1^{3/2}n}{10} + \frac{c_1^{1/2}c_2^{2/3}n^{2/3}}{10^{1/3}}.\]
When $n$ is sufficiently large, we may assume that 
\[ |\lines_2| < \frac{c_1^{3/2}n}{8}. \]

Ackerman \cite{Ackerman2019} proved that the $O(\cdot)$-notation in the bound of Theorem \ref{th:SzemTrot} may be replaced with the constant $2.44$.
This leads to 
\begin{equation} \label{eq:RichLinesInc} 
I(\pts,\lines_2) \le 2.44\cdot (|\pts|^{2/3}|\lines_2|^{2/3}+|\pts|+|\lines_2|) < \frac{2c_1n^{4/3}}{3} +O(n).
\end{equation}
 
We set $\lines' = \lines \setminus (\lines_1\cup\lines_2)$ and note every line of $\lines'$ is incident to $\Theta(n^{1/3})$ points of $\pts$. 
Combining \eqref{eq:PoorLinesInc} and \eqref{eq:RichLinesInc} leads to
\[ I(\pts,\lines') = I(\pts,\lines) - I(\pts,\lines_1) - I(\pts,\lines_2)> c_1n^{4/3}/12-o(n^{4/3}).\] 
 
Since every line of $\lines'$ is incident to $\Theta(n^{1/3})$ points of $\pts$, we have that 
\[ |\lines'|=\Theta(I(\pts,\lines')/n^{1/3}) = \Theta(n). \]
\end{proof}

We are now ready to prove Theorem \ref{th:StructuralSzemTrot}.
We first recall the statement of this result.
\vspace{2mm}

\noindent {\bf Theorem \ref{th:StructuralSzemTrot}.} $\qquad$\\
\emph{(a) For $1/3<\alpha< 1/2$, let $A,B\subset \RR$ satisfy $|A|=n^\alpha$ and $|B|=n^{1-\alpha}$.
Let $\lines$ be a set of $n$ lines in $\RR^2$, such that $I(A\times B,\lines) = \Theta(n^{4/3})$.
Then at least one of the following holds:
\begin{itemize}
\item There exists $1-2\alpha\le \beta \le 2/3$ such that $\lines$ contains $\Omega(n^{1-\beta}/\log n)$ families of $\Theta(n^\beta)$ parallel lines, each with a different slope.
\item There exists $1-\alpha\le \gamma \le 2/3$ such that $\lines$ contains $\Omega(n^{1-\gamma}/\log n)$ disjoint families of $\Theta(n^\gamma)$ concurrent lines.
\end{itemize}
(b) Assume that we are in the case of $\Omega(n^{1-\beta}/\log n)$ families of $\Theta(n^\beta)$ parallel lines. 
There exists $n^{2\beta} \le t \le n^{3\beta}$ such that, for $\Omega(n^{1-\beta}/\log^2 n)$ of these families, the additive energy of the $y$-intercepts is $\Theta(t)$.
Let $S$ be the set of slopes of these families. Then}
\[ E^{\times} (S) \cdot t = \Omega(n^{3-\alpha}/ \log^{12}n). \]

\begin{proof}
We remove from $\lines$ every line that is not incident to $\Theta(n^{1/3})$ points of $A\times B$.
By Lemma \ref{le:ThirdRichness}, this does not change the asymptotic size of $\lines$ and of $I(A\times B,\lines)$.
By definition, $\lines$ contains at most $n^\alpha+n^{1-\alpha}$ axis-parallel lines. 
These axis-parallel lines form at most $2n$ incidences with the points of $A\times B$.
We may thus discard these lines from $\lines$, without changing the asymptotic size of $|\lines|$ and $I(A\times B,\lines)$.
We write $I(A\times B,\lines)=cn^{4/3}+o(n^{4/3})$ for some positive $c\in \RR$.

\parag{An iterative process.}
We repeat the following process until $I(A\times B,\lines)< cn^{4/3}/2$.
Combining Theorem \ref{th:IncidencesViaLineEnergy} with the assumption $I(A\times B,\lines) = \Theta(n^{4/3})$ leads to
\begin{equation} \label{eq:n43VSenergy} 
n^{4/3} = O(|B|^{1/2}|A|^{2/3}E(\lines)^{1/6}|\lines|^{1/3} + |B|^{1/2}|\lines|). 
\end{equation}

Since $|B|=n^{1-\alpha}=o(n^{2/3})$, we get that $|B|^{1/2}|\lines|=o(n^{4/3})$. 
In other words, the right-hand side of \eqref{eq:n43VSenergy} is dominated by the first term.
We rewrite \eqref{eq:n43VSenergy} as
\[ n^{4/3} = O(n^{(1-\alpha)/2}n^{2\alpha/3}E(\lines)^{1/6}n^{1/3}) = O(n^{(5+\alpha)/6}E(\lines)^{1/6}).  \]
Rearranging gives
\[ E(\lines)^{1/6} = \Omega(n^{1/2-\alpha/6}),\]
which implies that
\begin{equation}\label{eq:StructureLinesLower}
 E(\lines) = \Omega(n^{3-\alpha}). 
\end{equation}

We imitate the notation of Theorem \ref{th:PrevEnergyBounds}: Let $M$ be the maximum number of lines of $\lines$ that are incident to the same point. Let $m$ be the maximum number of lines of $\lines$ that have the same slope.
Theorem \ref{th:PrevEnergyBounds} implies that 
\begin{equation}\label{eq:StructureLinesUpper}
E(\lines) = O\left(m^{1/2}n^{5/2} + Mn^2\right).
\end{equation}

Combining \eqref{eq:StructureLinesLower} and \eqref{eq:StructureLinesUpper} leads to 
\[ n^{3-\alpha} = O\left(m^{1/2}n^{5/2} + Mn^2\right). \]
Thus, we have that $m=\Omega(n^{1-2\alpha})$, or that $M=\Omega(n^{1-\alpha})$, or both.
When $m=\Omega(n^{1-2\alpha})$, there exists a set of $\Omega(n^{1-2\alpha})$ parallel lines of $\lines$.
When $M=\Omega(n^{1-\alpha})$, there exists a set of $\Omega(n^{1-\alpha})$ concurrent lines of $\lines$.
In either case, we remove those lines from $\lines$.
If we still have that $I(A\times B,\lines)\ge cn^{4/3}/2$, then we begin another iteration of the above process.

\parag{Studying the removed lines.}
Let $\lines'$ be the set of lines of $\lines$ that were removed in the iterative process above.
By definition, we have that $I(A\times B,\lines')> cn^{4/3}/2$.
Every line of $\lines'$ was removed as part of a family of $\Omega(n^{1-2\alpha})$ parallel lines or as part of a family of $\Omega(n^{1-\alpha})$ concurrent lines.
Let $\lines'_{\text{par}}$ be the set of the former lines and let $\lines'_{\text{con}}$ be the set of the latter lines.
We note that $I(A\times B,\lines'_{\text{par}})> cn^{4/3}/4$ or that $I(A\times B,\lines'_{\text{con}})> cn^{4/3}/4$ (or both).

We first assume that $I(A\times B,\lines'_{\text{par}})> cn^{4/3}/4$.
For $\log_2 n^{1-2\beta}\le j\le \log_2 n$, let $\lines'_{\text{par},j}$ be the set of lines that joined $\lines'$ as part of a family of at least $2^j$ parallel lines, and of fewer than $2^{j+1}$ such lines.
By the pigeonhole principle, there exists a $j$ for which $I(A\times B,\lines'_{\text{par},j}) = \Omega(n^{4/3}/\log n)$.
We denote one such $\lines'_{\text{par},j}$ as $\lines''$.
Since every line is incident to $\Theta(n^{1/3})$ points of $A\times B$, we have that $|\lines''|=\Omega(n/\log n)$.
By writing $\beta = \log_n 2^j$, we get that there exist $\Omega(n^{1-\beta}/\log n)$ families of $\Theta(n^\beta)$ parallel lines in $\lines$.  

Consider a family of $\Theta(n^\beta)$ parallel lines from $\lines''$.
Since these lines are parallel, every point of $A\times B$ is incident to at most one line.
Thus, the family forms at most $n$ incidences with $A\times B$.
Since $\lines''$ consists of $O(n^{1-\beta})$ such families, we get that $I(A \times B,\lines'')=O(n^{2-\beta})$.
Since $I(A\times B,\lines'')=\Omega(n^{4/3}/\log n)$, we conclude that $\beta\le 2/3$.

Next, we assume that $I(A\times B,\lines'_{\text{par}})\le cn^{4/3}/4$.
In this case, $I(A\times B,\lines'_{\text{con}})> cn^{4/3}/4$.
By repeating the argument from the parallel case, we obtain $1-\alpha\le \gamma \le 1$ such that lines from families of $\Theta(n^\gamma)$ concurrent lines form $\Omega(n^{4/3}/\log n)$ incidences.
Let $\lines''$ be the set of lines of $\lines'_{\text{con}}$ that belong to such a family. 
Since every line of $\lines''$ is incident to $\Theta(n^{1/3})$ points of $A\times B$, we have that $|\lines''|=\Omega(n/\log n)$.

Consider a family of $\Theta(n^\gamma)$ concurrent lines from $\lines''$.
Excluding the point of intersection, every point of $A\times B$ is incident to at most one such line.
Thus, the family forms at most $n+\Theta(n^\gamma)$ incidences.
Since $\lines''$ consists of $O(n^{1-\gamma})$ such families, we get that $I(A \times B,\lines'')=O(n^{2-\gamma})$.
Since $I(A\times B,\lines'')=\Omega(n^{4/3}/\log n)$, we conclude that $\gamma\le 2/3$.

\parag{Proof of part (b).}
In this part, we assume that $\lines''$ consists of $\Omega(n^{1-\beta}/\log n)$ families of $\Theta(n^\beta)$ parallel lines. 
Let $S''$ be the set of slopes of the lines in $\lines''$.  
Then $|S''|=\Omega(n^{1-\beta}/\log n)$ and $|S''|=O(n^{1-\beta})$.
For $s\in S''$, we set $Y_{\lines''}^s$ to be the set of $y$-intercepts of the lines with slope $s$ in $\lines''$.

We recall that $I(A\times B,\lines'')=\Omega(n^{4/3}/\log n)$.
By definition, for every $s\in S''$, we have that $|Y_{\lines''}^s|=\Theta(n^\beta)$. 
This implies that $E^+(Y_{\lines''}^s) =\Omega(n^{2\beta})$ and $E^+(Y_{\lines''}^s) =O(n^{3\beta})$.
For $n^{2\beta}\le t\le n^{3\beta}$, we define
\[ \lines_t = \{\ell \in \lines''\ :\ E^+( Y_{\lines''}^s) =\Theta(t)\}. \]
By another dyadic pigeonhole argument, there exists a $t$ that satisfies $I(A\times B,\lines_t) = \Omega(n^{4/3}/\log^2 n)$.
We fix such a value of $t$. 
Since every line is incident to $\Theta(n^{1/3})$ points of $A\times B$, we get that $|\lines_t|=\Omega(n/\log^2 n)$.

Let $S$ be the set of slopes of lines of $\lines_t$.
We have that 
\[ |S| \le |S''| =O(n^{1-\beta}) \quad \text{ and that } \quad |S|=\Omega(n^{1-\beta}/\log^2 n). \] 

To obtain an upper bound for $E(\lines_t)$, we imitate the proof of Theorem \ref{th:LineCartesianEnergyUpper}.
While $\lines_t$ may not be a Cartesian product of lines, we can instead rely on the property $E^+( Y_{\lines''}^s) =\Theta(t)$.
We think of $E(\lines_t)$ as the number of solutions to \eqref{eq:LineQuadrupleCond1} and \eqref{eq:LineQuadrupleCond2}.

Equation \eqref{eq:LineQuadrupleCond1} depends only on the elements of $S$.
The number of distinct solutions to \eqref{eq:LineQuadrupleCond1} from $S$ is $\Theta(E^\times(S))$.
This does not take into account the number of lines that have the same slope. 
We fix a solution to \eqref{eq:LineQuadrupleCond1} and consider the number of corresponding solutions to \eqref{eq:LineQuadrupleCond2}.
In other words, we fix the slopes $c_1,c_2,c_3,c_4\in S$ and consider valid values for the $y$-intercepts $d_1,d_2,d_3,d_4$.
We set $s = c_1/c_3$ and rephrase \eqref{eq:LineQuadrupleCond2} as 
\begin{equation} \label{eq:LineQuadrupleCond2structure} 
d_2-d_1 = s(d_4-d_3).
\end{equation}

Similarly to $r_A^-(x)$, for sets $A$ and $B$ we define 
\begin{align*} 
r_{A,B}^-(x) &= |\{(a,b)\in A\times B\ :\ a-b=x\}|.
\end{align*}

By the Cauchy-Schwarz inequality, the number of solutions to \eqref{eq:LineQuadrupleCond2structure} is 
\begin{align*} 
\sum_{k\in \RR}r_{Y_{\lines_t}^{c_1},Y_{\lines_t}^{c_2}}^-(k) r_{Y_{\lines_t}^{c_3},Y_{\lines_t}^{c_4}}^-(k/s) &\le \left(\sum_{k\in \RR}r_{Y_{\lines_t}^{c_1},Y_{\lines_t}^{c_2}}^-(k)^2\right)^{1/2} \left(\sum_{k\in \RR}r_{Y_{\lines_t}^{c_3},Y_{\lines_t}^{c_4}}^-(k)\right)^{1/2} \\[2mm]
&\hspace{40mm} = E^+(Y_{\lines_t}^{c_1},Y_{\lines_t}^{c_2})^{1/2} \cdot E^+(Y_{\lines_t}^{c_3},Y_{\lines_t}^{c_4})^{1/2}.
\end{align*}
Combining this with \eqref{eq:BipartiteEnergy} implies that the number of solutions to \eqref{eq:LineQuadrupleCond2structure} is at most
\[ E^+(Y_{\lines_t}^{c_1})^{1/4}\cdot E^+(Y_{\lines_t}^{c_2})^{1/4}\cdot E^+(Y_{\lines_t}^{c_3})^{1/4}\cdot E^+(Y_{\lines_t}^{c_4})^{1/4}. \]
By definition, each of these four energies is $\Theta(t)$.
We conclude that the number of solutions to \eqref{eq:LineQuadrupleCond2structure} is $O(t)$.

To recap, there are $\Theta(E^\times(S))$ solutions to \eqref{eq:LineQuadrupleCond1}, each with $O(t)$ corresponding solutions to \eqref{eq:LineQuadrupleCond2structure}.
We thus have that
\begin{equation} \label{eq:EnergyLijUpper}
    E(\lines_t) = O(E^{\times} (S) \cdot t).
\end{equation}

Combining Theorem \ref{th:IncidencesViaLineEnergy} with $I(\pts,\lines_t)=\Omega(n^{4/3}/\log^2 n)$ leads to
\begin{equation} \label{eq:n43VSenergylog2} 
n^{4/3}/ \log^2(n) = O(|B|^{1/2}|A|^{2/3}E(\lines_t)^{1/6}|\lines_t|^{1/3} + |B|^{1/2}|\lines_t|). 
\end{equation}

Since $|B|=n^{1-\alpha}$ and $\alpha>1/3$, we have that $|B|^{1/2}|\lines_t|=o(n^{4/3}/ \log^2n)$. 
This implies that the right-hand side of \eqref{eq:n43VSenergylog2} is dominated by the first term.
We rewrite \eqref{eq:n43VSenergylog2} as
\[ n^{4/3}/ \log^2n = O(n^{(1-\alpha)/2}n^{2\alpha/3}E(\lines_t)^{1/6}n^{1/3}) = O(n^{(5+\alpha)/6}E(\lines_t)^{1/6}).  \]
Rearranging gives
\[ E(\lines_t)^{1/6} = \Omega(n^{1/2-\alpha/6}/ \log^2n),\]
which implies that
\begin{equation*}
 E(\lines_t) = \Omega(n^{3-\alpha}/ \log^{12}n). 
\end{equation*}

Combining the bound above with \eqref{eq:EnergyLijUpper} leads to 
\[ E^{\times} (S) \cdot t = \Omega(n^{3-\alpha}/ \log^{12}n). \]
\end{proof}

\section{Elekes's structural problem} \label{sec:ElekStructural}

This concluding section contains the proof of Theorem \ref{th:ElekesStructural}.
We first recall the statement of this result.
\vspace{2mm}

\noindent {\bf Theorem \ref{th:ElekesStructural}.}
\emph{Let $A$ be a set of $n$ reals and let $\alpha<1$ satisfy $\alpha = \Omega(n^{-1/2}\log^2 k)$.
Let $\lines$ be a set of $k$ non-axis-parallel lines, each incident to at least $\alpha n$ points of $A\times A$.
Consider $0\le\beta\le 1$ such that $\lines$ contains $\Omega(k^{1-\beta}/\log k)$ families of $\Theta(k^\beta)$ parallel lines.
There exists $k^{2\beta} \le t \le k^{3\beta}$ such that, for $\Omega(k^{1-\beta}/\log^2 k)$ of these families, the additive energy of the $y$-intercepts is $\Theta(t)$.
Let $S$ be the set of slopes of these families. 
Then }
\[ k = O\left(\frac{n^{1/4}t^{1/4}E^\times(S)^{1/4}\log^{3} n}{\alpha^{3/2}}\right). \]

\begin{proof}
We imitate the proof of Theorem \ref{th:StructuralSzemTrot}(b).
Let $\lines'$ be the set of lines that belong to the $\Omega(k^{1-\beta}/\log n)$ families of $\Theta(k^\beta)$ parallel lines. 
Let $S'$ be the set of slopes of the lines of $\lines'$.  
Then $|S'|=\Omega(k^{1-\beta}/\log n)$ and $|S'|=O(k^{1-\beta})$.
For $s\in S'$, we denote as $Y_{\lines'}^s$ the set of $y$-intercepts of the lines of $\lines'$ with slope $s$.

By definition, for every $s\in S'$, we have that $|Y_{\lines'}^s|=\Theta(k^\beta)$. 
This implies that $E^+(Y_{\lines'}^s) =\Omega(k^{2\beta})$ and that $E^+(Y_{\lines'}^s) =O(k^{3\beta})$.
For $n^{2\beta}\le t\le n^{3\beta}$, we define
\[ \lines_t = \{\ell \in \lines'\ :\ E^+( Y_{\lines'}^s) =\Theta(t)\}. \]
By the pigeonhole principle, there exists $t$ that satisfies $|\lines_t| = \Omega(k/\log^2 k)$.
We fix such a value of $t$.

Since every line of $\lines$ is incident to at least $\alpha n$ points of $A\times A$, we get that
\begin{equation*}
I(A\times A,\lines_t) =\Omega(\alpha n k/\log^2 k). 
\end{equation*}

Theorem \ref{th:IncidencesViaLineEnergy} implies that
\[ I(A\times A,\lines_t)= O(n^{7/6}E(\lines_t)^{1/6}|\lines_t|^{1/3}+|\lines_t|n^{1/2}) = O(n^{7/6}E(\lines_t)^{1/6}k^{1/3}+kn^{1/2}).\] 
Combining the two above bounds for $I(A\times A,\lines_t)$ leads to
\[ \alpha n k/\log^2 k =O(n^{7/6}E(\lines_t)^{1/6}k^{1/3}+kn^{1/2}). \]

Since $\alpha = \Omega(n^{-1/2} \log^2 k)$, we have that
\begin{equation} \label{eq:EnergyLowerAddMultEnergyElekStruct} 
\alpha n k/\log^2 k =O(n^{7/6}E(\lines_t)^{1/6}k^{1/3}). 
\end{equation}

By repeating the argument from the proof of Theorem \ref{th:StructuralSzemTrot}(b), we obtain that 
\[ E(\lines_t) = O(t\cdot  E^\times(S)). \]
Combining this with \eqref{eq:EnergyLowerAddMultEnergyElekStruct} and noting that $\log k = O(\log n)$ leads to
\begin{align*} 
\alpha n k =O(n^{7/6}E(\lines_t)^{1/6}k^{1/3}\log^2 k) &= O(n^{7/6}k^{1/3} t^{1/6}E^\times(S)^{1/6}\log^2 n).
\end{align*}
Rearranging leads to
\[ k = O\left(\frac{n^{1/4}t^{1/4}E^\times(S)^{1/4}\log^{3} n}{\alpha^{3/2}}\right).\]
\end{proof}

\bibliographystyle{plain}

\bibliography{LineEnergy}

\end{document}